\documentclass[12pt,a4paper,notitlepage,oneside]{article}
\usepackage{latexsym}
\usepackage{array}
\usepackage{amsmath, amsthm,amssymb,amscd,amstext}
\usepackage{amsfonts}
\usepackage{epic}
\usepackage{graphicx}
\usepackage{enumerate}
\input xy
\xyoption{all}
\usepackage[left=3.0cm,right=3.0cm,top=2.0cm,bottom=2.0cm]{geometry}

\renewcommand*{\mod}{\mathrm{mod}\,}
\newcommand{\ind}{\mathrm{ind}\,}
\newcommand{\Hom}{\mathrm{Hom}}
\newcommand{\Tr}{\mathrm{Tr}}
\newcommand{\op}{\mathrm{op}}
\newcommand{\End}{\mathrm{End}}
\newcommand{\Ext}{\mathrm{Ext}}
\newcommand{\rad}{\mathrm{rad}\,}
\newcommand{\add}{\mathrm{add}}
\newcommand{\Tor}{\mathrm{Tor}}
\newcommand{\ann}{\mathrm{ann}}

\makeatletter
\renewcommand{\@seccntformat}[1]{\normalsize{\csname
the#1\endcsname.\hspace{0,3cm}}} \makeatother

\newtheorem{thm}{\sc Theorem}[section]
\newtheorem{cor}[thm]{\sc Corollary}
\newtheorem{lem}[thm]{\sc Lemma}
\newtheorem{prop}[thm]{\sc Proposition}

\makeatletter
\renewcommand\@biblabel[1]{#1.}
\makeatother

\title{{ \large CLASSIFICATION OF MODULES NOT LYING\\ ON SHORT CHAINS}}
\author{\small  ALICJA JAWORSKA, PIOTR MALICKI, AND ANDRZEJ SKOWRO{\'N}SKI}
\date{}
\begin{document}
\bibstyle{numeric}
\maketitle

\begin{abstract}
\noindent We  give a complete description of finitely generated modules over artin algebras which are
not the middle of a short chain of modules, using injective and tilting modules over hereditary artin algebras.
\end{abstract}

\section{\normalsize Introduction}
Throughout the paper, by an algebra we mean an artin algebra over
a fixed commutative artin ring $R$, that is, an $R$-algebra
(associative, with identity) which is finitely generated as an
$R$-module. For an algebra $A$, we denote by $\mod A$ the category
of finitely generated right $A$-modules, by $\ind A$ the full
subcategory of $\mod A$ formed by the indecomposable modules, by
$K_0(A)$ the Grothendieck group of $A$, and by $[M]$ the image of
a module $M$ from $\mod A$ in $K_0(A)$. Then $[M]=[N]$ for two
modules $M$ and $N$ in $\mod A$ if and only if $M$ and $N$  have
the same (simple) composition factors including the
multiplicities. A module $M$ in $\mod A$ is called sincere if
every simple right $A$-module occurs as a composition factor of
$M$. Further, we denote by $D$ the standard duality $\Hom_R(-,E)$
on $\mod A$,
 where $E$ is a minimal injective cogenerator in $\mod R$. Moreover, for a module $X$ in $\mod A$ and its minimal projective
 presentation $\xymatrix@C=13pt{P_1 \ar[r]^f &  P_0 \ar[r] &X \ar[r] &0}$ in $\mod A$, the transpose $\Tr X$ of $X$ is the
cokernel of the homomorphism $\Hom_A(f, A)$ in $\mod A^{\op}$,
where $A^{\op}$ is the opposite algebra of $A$. Then we obtain the
homological operator $\tau_A=D\Tr$ on modules in $\mod A$, called
the Auslander-Reiten translation, playing a fundamental role in
the modern representation theory of artin algebras.

The aim of this article is to provide a complete description of all modules $M$ in $\mod A$ satisfying the condition:
 for  any module $X$ in $\ind A$, we have $\Hom_A(X,M)=0$ or $\Hom_A(M, \tau_AX)=0$. We note that, by \cite{AR},
 \cite{RSS}, a sequence $\xymatrix@C=13pt{X \ar[r] & M \ar[r] & \tau_AX}$ of nonzero
homomorphisms in  $\mod A$ with $X$ being indecomposable is called
a short chain, and $M$ the middle of this short chain. Therefore,
we are concerned with the classification of all modules in $\mod
A$ which are not the middle of a short chain. We also mention
that, if $M$ is a module in $\mod A$ which is not the middle of a
short chain, then $\Hom_A(M, \tau_AM)=0$, and hence the number of
pairwise nonisomorphic indecomposable direct summands of $M$ is
less than or equal to the rank of $K_0(A)$, by \cite[Lemma 2]{S3}.
Further, by \cite[Theorem 1.6]{RSS} and \cite[Lemma 1]{HL}, an
indecomposable module $X$ in $\mod A$ is not the middle of a short
chain if and only if $X$ does not lie on a short cycle
$\xymatrix@C=13pt{Y \ar[r] & X \ar[r] & Y}$ of nonzero
nonisomorphisms in $\ind A$. Hence, every indecomposable direct
summand $Z$ of a module $M$ in $\mod A$ which is not the middle of
a short chain is uniquely determined (up to isomorphism) by the
composition factors (see \cite[Corollary 2.2]{RSS}). Finally, we
point out that the class of modules which are not the middle of a
short chain contains the class of directing modules investigated
in \cite{Bak}, \cite{HRi2}, \cite{Ri1}, \cite{S3}, \cite{S4},
\cite{SW}.

Following \cite{Bon}, \cite{HRi1}, by a tilted algebra we mean an
algebra of the form $\End_H(T)$, where $H$ is a hereditary algebra
and $T$ is a tilting module in $\mod H$, that is,
$\Ext^1_H(T,T)=0$ and the number of pairwise nonisomorphic
indecomposable direct summands of $T$ is equal to the rank of
$K_0(H)$. The tilted algebras play  a prominent role in the
representation theory of algebras and have attracted much
attention
(see \cite{ASS}, \cite{Ha}, \cite{MS}, \cite{Ri1}, \cite{Ri2}, \cite{SS1}, \cite{SS2} and their cited papers).\\

The following theorem is the main result of the paper.
\begin{thm} \label{thm 1.1}
Let $A$ be an algebra and $M$ a module in $\mod A$ which is not
the middle of a short chain. Then there exists a hereditary
algebra $H$, a tilting module $T$ in $\mod H$, and an injective
module $I$ in $\mod H$ such that the following statements hold:
\begin{enumerate}
\renewcommand{\labelenumi}{\rm(\roman{enumi})}
\item the tilted algebra $B=\End_H(T)$ is a quotient algebra of $A$;
\item $M$ is isomorphic to the right $B$-module $\Hom_H(T,I)$.
\end{enumerate}
\end{thm}

We note that for  a hereditary algebra $H$, $T$ a tilting module
in $\mod H$, $I$ an injective module in $\mod H$, and
$B=\End_H(T)$, the right $B$-module $\Hom_H(T,I)$ is not the
middle of a short chain in $\mod B$ (see Lemma \ref{lem 3.1}). An
important role in the proof of the main theorem plays the
following characterization of tilted algebras established recently
in the authors paper \cite{JMS1}: an algebra $B$ is a tilted
algebra if and only if $\mod B$ admits a sincere module $M$ which
is not the middle of a short chain.\\

The following fact is a consequence of Theorem \ref{thm 1.1}.
\begin{cor} \label{cor 1.2}
Let $A$ be an algebra and $M$ a module in $\mod A$ which is not
the middle of a short chain. Then $\End_A(M)$ is a hereditary
algebra.
\end{cor}

In Sections 2 and 3, after recalling some background on module
categories  and tilted algebras, we prove preliminary facts
playing an essential role in the proof of Theorem \ref{thm 1.1}.
Section 4 is devoted to the proofs of Theorem \ref{thm 1.1} and
Corollary \ref{cor 1.2}. In the final Section 5 we present
examples illustrating the main
theorem.\\

For background on the representation theory applied here we refer
to \cite{ASS}, \cite{ARS}, \cite{Ri1}, \cite{SS1}, \cite{SS2}.


\section{\normalsize Preliminaries on module categories}

Let $A$ be an algebra. We  denote by $\Gamma_A$ the
Auslander-Reiten quiver of $A$. Recall that $\Gamma_A$ is a valued
translation quiver whose vertices are the isomorphism classes
$\{X\}$ of modules $X$ in $\ind A$, the valued arrows of
$\Gamma_A$ correspond to irreducible homomorphisms between
indecomposable modules (and describe minimal left almost split
homomorphisms with indecomposable domains and minimal right almost
split homomorphisms with indecomposable codomains) and  the
translation is given by the Auslander-Reiten translations
$\tau_A=D\Tr$ and $\tau^-_A=\Tr D$. We shall not distinguish
between a module $X$ in $\ind A$ and the corresponding vertex
$\{X\}$ of $\Gamma_A$. By a component of $\Gamma_A$ we mean a
connected component of the quiver $\Gamma_A$. Following \cite{S2},
a component  $\mathcal{C}$ of $\Gamma_A$ is said to be generalized
standard if $\rad^{\infty}_A(X,Y)=0$
 for all modules $X$ and $Y$ in $\mathcal{C}$, where
 $\rad^{\infty}_A$ is the infinite Jacobson radical of $\mod A$.
 Moreover, two components $\mathcal{C}$ and $\mathcal{D}$ of
 $\Gamma_A$ are said to be orthogonal if $\Hom_A(X,Y)=0$ and
 $\Hom_A(Y,X)=0$ for all modules $X$ in $\mathcal{C}$ and $Y$ in
 $\mathcal{D}$. A family $\mathcal{C}=(\mathcal{C}_i)_{i \in I}$
 of components of $\Gamma_A$ is said to be (strongly) separating
 if the components in $\Gamma_A$ split into three disjoint
 families $\mathcal{P}^A$, $\mathcal{C}^A= \mathcal{C}$ and
 $\mathcal{Q}^A$ such that the following conditions are
 satisfied:
 \begin{itemize}
 \item[(S1)] $ \mathcal{C}^A$ is a sincere family of pairwise orthogonal
 generalized standard components;
 \item[(S2)] $\Hom_A(\mathcal{Q}^A,
 \mathcal{P}^A)=0$, $\Hom_A(\mathcal{Q}^A, \mathcal{C}^A)=0$,
 $\Hom_A(\mathcal{C}^A,
 \mathcal{P}^A)=0$;
 \item [(S3)] any homomorphism from $\mathcal{P}^A$ to $\mathcal{Q}^A$ in
 $\mod A$ factors through $\add (\mathcal{C}_i)$ for any
 $i \in I$.
 \end{itemize}
We then say that $\mathcal{C}^A$ separates $\mathcal{P}^A$ from
$\mathcal{Q}^A$ and write
$$\Gamma_A=\mathcal{P}^A \vee
\mathcal{C}^A \vee \mathcal{Q}^A.$$
 A component $\mathcal{C}$ of
$\Gamma_A$ is said to be preprojective if $\mathcal{C}$ is acyclic
(without oriented cycles) and each module in $\mathcal{C}$ belongs
to the $\tau_A$-orbit of a projective module. Dually,
$\mathcal{C}$ is said to be preinjective if $\mathcal{C}$ is
acyclic and each module in $\mathcal{C}$ belongs to the
$\tau_A$-orbit of an injective module. Further, $\mathcal{C}$ is
called regular if $\mathcal{C}$ contains neither a projective
module nor an injective module. Finally, $\mathcal{C}$ is called
semiregular if $\mathcal{C}$ does not contain both a projective
module and an injective module. By a general result of S. Liu
\cite{L0} and Y. Zhang \cite{Z}, a regular component $\mathcal{C}$
contains an oriented cycle if and only if $\mathcal{C}$ is a
stable tube, that is, an orbit  quiver
$\mathbb{ZA}_{\infty}/(\tau^r)$, for some integer $r \geq 1$.
Important classes of semiregular components with oriented cycles
are formed by the ray tubes, obtained from stable tubes by a
finite number (possibly empty) of ray insertions, and the coray
tubes obtained from stable tubes by a finite number (possibly
empty) of coray insertions (see \cite{Ri1}, \cite{SS2}).

The following characterizations of ray and coray tubes of
Auslander-Reiten quivers of algebras have been established by S.
Liu in \cite{L1a}.
\begin{thm}
Let $A$ be an algebra and $\mathcal{C}$ be a semiregular component
of $\Gamma_A$. The following equivalences hold:
\begin{enumerate}
\renewcommand{\labelenumi}{\rm(\roman{enumi})}
\item  $\mathcal{C}$ contains an oriented cycle but no injective
module if and only if $\mathcal{C}$ is a ray tube;
\item $\mathcal{C}$ contains an oriented cycle but no projective
module if and only if $\mathcal{C}$ is a coray tube.
\end{enumerate}
\end{thm}

The following lemma from \cite[Lemma 1.2]{JMS1} will play an
important role in the proof of our main theorem.

\begin{lem}\label{lem 2.2}
Let $A$ be an algebra and $M$ a sincere module in $\mod A$ which
is not the middle of a short chain. Then the following statements
hold:
\begin{enumerate}
\renewcommand{\labelenumi}{\rm(\roman{enumi})}
\item $\Hom_A(M,X)=0$ for any $A$-module $X$ in $\mathcal{T}$, where $\mathcal{T}$ is an arbitrary ray tube of $\Gamma_A$ containing a projective
module;
\item $\Hom_A(X,M)=0$ for any $A$-module $X$ in $\mathcal{T}$, where $\mathcal{T}$ is an arbitrary coray tube of $\Gamma_A$ containing an injective module.
\end{enumerate}
\end{lem}

\begin{lem} \label{lem 2.3}
Let $A$ be an algebra, $\mathcal{C}=(\mathcal{C}_i)_{i \in I}$ a
separating family of stable tubes of $\Gamma_A$, and $\Gamma_A =
\mathcal{P}^A \vee \mathcal{C}^A \vee \mathcal{Q}^A$ the
associated decomposition of $\Gamma_A$ with
$\mathcal{C}^A=\mathcal{C}$. Then for arbitrary modules $M \in
\mathcal{P}^A$, $N \in \mathcal{Q}^A$, and $i \in I$, the
following statements hold:
\begin{enumerate}
\renewcommand{\labelenumi}{\rm(\roman{enumi})}
\item $\Hom_A(M,X) \neq 0$ for all but finitely many modules $X \in
\mathcal{C}_i$;
\item $\Hom_A(X, N) \neq 0$ for all but finitely many modules $X \in
\mathcal{C}_i$.
\end{enumerate}
\end{lem}
\begin{proof}
Let $M$ be a module in $\mathcal{P}^A$, $N$ a module in
$\mathcal{Q}^A$, $i \in I$, and $r_i$ be the rank of the stable
tube $\mathcal{C}_i$. Consider an injective hull $M \rightarrow
E_A(M)$ of $M$ in $\mod A$ and a projective cover
$P_A(N)\rightarrow N$ of $N$ in $\mod A$. Applying the separating
property of $\mathcal{C}$, we conclude that there exist
indecomposable modules $U$ and $V$ in $\mathcal{C}_i$ such that
$\Hom_A(M,U) \neq 0$ and $\Hom_A(V,N) \neq 0$. Then $\Hom_A(M,X)
\neq 0$ and $\Hom_A(X,N)\neq 0$ for all indecomposable modules $X$
in $\mathcal{C}_i$ of quasi-length greater than or equal to $r_i$,
by \cite[Lemma 3.9]{S2a}. Since such modules $X$ exhaust all but
finitely many modules in $\mathcal{C}_i$, the claims (i) and (ii)
hold.
\end{proof}

We also have the following known fact.
\begin{lem} \label{lem 2.4}
Let $A$ be an algebra and $\mathcal{T}$ a stable tube of
$\Gamma_A$. Then every indecomposable module $X$ in $\mathcal{T}$
is the middle of a short chain in $\mod A$.
\end{lem}

A path $\xymatrix@C=13pt{X_0 \ar[r] & X_1 \ar[r] & ... \ar[r]
&X_{t-1} \ar[r] & X_t}$ in the Auslander-Reiten quiver $\Gamma_A$
of an algebra $A$ is called sectional if $\tau_AX_i \ncong
X_{i-2}$ for all $i \in \{2,...,t\}$. Then we have the following
result proved by R. Bautista and S. O. Smal\o \, \cite{BS}.
\begin{lem} \label{lem 2.5}
Let $A$ be an algebra and
\[\xymatrix@C=13pt{X_0 \ar[r]^{f_1} & X_1
\ar[r]^{f_2} & ... \ar[r] &X_{t-1} \ar[r]^{f_t} & X_t}\] be a path
of irreducible homomorphisms $f_1,f_2,..., f_t$ corresponding to a
sectional path of $\Gamma_A$. Then $f_t...f_2f_1 \neq 0$.
\end{lem}

Let $A$ be an algebra, $\mathcal{C}$ a component of $\Gamma_A$ and
$V$, $W$ be $A$-modules in $\mathcal{C}$ such that $V$ is a
predecessor of $W$ (respectively, a successor of $W$). If $V$ lies
on a sectional path from $V$ to $W$ (respectively, from $W$ to
$V$), then we say that $V$ is a sectional predecessor of $W$
(respectively, a sectional successor of $W$). Otherwise, we say
that $V$ is a nonsectional predecessor of $W$ (respectively, a
nonsectional successor of $W$). Moreover, denote by
$\mathcal{S}_W$ the set of all indecomposable modules $X$ in
$\mathcal{C}$ such that there is a sectional path in $\mathcal{C}$
(possibly of length zero) from $X$ to $W$, and by
$\mathcal{S}^*_W$ the set of all indecomposable modules $Y$ in
$\mathcal{C}$ such that there is a sectional path in $\mathcal{C}$
(possibly of length zero) from $W$ to $Y$.

\begin{prop}\label{prop 2.6}
Let $A$ be an algebra and $\mathcal{C}$ be an acyclic component of
$\Gamma_A$ with finitely many $\tau_A$-orbits. Then the following
statements hold:
\begin{enumerate}
\renewcommand{\labelenumi}{\rm(\roman{enumi})}
\item if $V$ and $W$ are modules in $\mathcal{C}$ such that $V$ is a predecessor of $W$, $V$ does not belong to $\mathcal{S}_W$, and $W$ has no injective
nonsectional predecessors in $\mathcal{C}$, then we have
$\Hom_A(V,\tau_AU)\neq 0$ for some module $U$ in $\mathcal{S}_W$;
\item if $V$ and $W$ are modules in $\mathcal{C}$ such that $V$ is a successor of $W$, $V$ does not belong to $\mathcal{S}^*_W$, and $W$ has no projective
nonsectional successors in $\mathcal{C}$, then we have
$\Hom_A(\tau^-_AU,V)\neq 0$ for some module $U$ in
$\mathcal{S}^*_W$.
\end{enumerate}
\end{prop}

\begin{proof}
We shall prove only (i), because the proof of (ii) is dual. Let
$V$ and $W$ be modules in $\mathcal{C}$ such that $V$ is a
predecessor of $W$, $V$ does not belong to $\mathcal{S}_W$, and
$W$ has no injective nonsectional predecessors in $\mathcal{C}$.
Moreover, let $n(V)$ be the length of the shortest path in
$\mathcal{C}$ from $V$ to $W$. We prove first by induction on
$n(V)$ that then every path in $\mathcal{C}$ of sufficiently large
length starting at $V$ is passing through a module in
$\tau_A\mathcal{S}_W$.

We may assume that $V$ does not belong to $\tau_A\mathcal{S}_W$
and hence $n(V)\geq 3$. Because $W$ has no injective nonsectional
predecessors in $\mathcal{C}$ and $\mathcal{S}_W$ does not contain
the module $V$, we conclude that there exists $\tau_A^-V$ and it
is a predecessor of $W$ in $\mathcal{C}$. Moreover,
$n(\tau_A^-V)=n(V)-2$. Indeed, if it is not the case, then we get
a contradiction with the minimality of $n(V)$. Let $\{U_1, U_2,
\ldots, U_t\}$ be the set of all direct predecessors of
$\tau_A^-V$ in $\mathcal{C}$. Then, for any $i\in\{1,\ldots,t\}$,
$U_i$ is a predecessor of $W$ in $\mathcal{C}$ and
$n(U_i)=n(V)-1$. Hence, by the induction hypothesis, every path of
sufficiently large length starting at $U_i$ is passing through  a
module in  $\tau_A\mathcal{S}_W$. Since $\{U_1, U_2, \ldots,
U_t\}$ is also the set of all direct successors of $V$, we have
that every path in $\mathcal{C}$ of nonzero length starting at $V$
is passing through $U_i$ for some $i\in\{1,\ldots,t\}$. Therefore,
the required property holds.

Let now $u: V\to E_A(V)$ be an injective hull of $V$ in $\mod A$.
Then there exists an indecomposable injective $A$-module $I$ such
that $\Hom_A(V,I)\neq 0$. Since $W$ has no injective nonsectional
predecessors in $\mathcal{C}$, applying \cite[Chapter IV, Lemma
5.1]{ASS}, we conclude that there exists a path of irreducible
homomorphisms
$$\xymatrix{V=V_0 \ar[r]^(.55){g_1}& V_1 \ar[r]^(.45){g_2} & V_2 \ar[r]^(.45){}&\cdots\ar[r]^(.45){}& V_{r-1} \ar[r]^(.5){g_{r}} & V_r} $$
with $V_r=\tau_AU$ for some $U\in\mathcal{S}_W$ and a homomorphism
$h_r: V_r\to I$ such that $h_rg_r\ldots g_1\neq 0$. Hence, we
conclude that $\Hom_A(V,\tau_AU)=\Hom_A(V,V_r)\neq 0$.
\end{proof}


\section{\normalsize Preliminaries on tilted algebras}

Let $H$ be an indecomposable hereditary algebra and $Q_H$ the
valued quiver of $H$. Recall that the vertices of $Q_H$ are the
numbers $1, 2, \ldots, n$ corresponding to a complete set $S_1,
S_2, \ldots, S_n$ of pairwise nonisomorphic simple modules in
$\mod H$ and there is an arrow from $i$ to $j$ in $Q_H$ if
$\Ext^1_H(S_i,S_j)\neq 0$, and then to this arrow is assigned the
valuation
$(\dim_{\End_H(S_j)}\Ext^1_H(S_i,S_j),\dim_{\End_H(S_i)}\Ext^1_H(S_i,S_j))$.
Recall that the Auslander-Reiten quiver $\Gamma_H$ of $H$ has a
disjoint union decomposition of the form
\[\Gamma_H = \mathcal{P}(H) \vee \mathcal{R}(H) \vee \mathcal{Q}(H),\]
where $\mathcal{P}(H)$ is the preprojective component containing
all indecomposable projective $H$-modules, $\mathcal{Q}(H)$ is the
preinjective component containing all indecomposable injective
$H$-modules, and $\mathcal{R}(H)$ is the family of all regular
components of $\Gamma_H$. More precisely, we have:
\begin{itemize}
\item[$\bullet$] if $Q_H$ is a Dynkin quiver, then $\mathcal{R}(H)$ is empty and
$\mathcal{P}(H)=\mathcal{Q}(H)$;
\item[$\bullet$] if $Q_H$ is a Euclidean quiver, then $\mathcal{P}(H)\cong (-\mathbb{N})Q^{\op}_H$, $\mathcal{Q}(H)\cong \mathbb{N}Q^{\op}_H$ and
$\mathcal{R}(H)$ is a separating infinite family of stable tubes;
\item[$\bullet$] if $Q_H$ is a wild quiver, then $\mathcal{P}(H) \cong (-\mathbb{N})Q^{\op}_H$, $\mathcal{Q}(H)\cong \mathbb{N}Q^{\op}_H$ and
$\mathcal{R}(H)$ is an infinite family of components of type
$\mathbb{ZA}_{\infty}$.
\end{itemize}
Let $T$ be a tilting module in $\mod H$ and $B=\End_H(T)$ the
associated tilted algebra. Then the tilting $H$-module $T$
determines the torsion pair $(\mathcal{F}(T), \mathcal{T}(T))$ in
$\mod H$, with the torsion-free part $\mathcal{F}(T)=\{X \in \mod
H | \Hom_H(T,X)=0\}$ and the torsion part $\mathcal{T}(T)=\{X \in
\mod H | \Ext^1_H(T,X)=0\}$, and the splitting torsion pair
$(\mathcal{Y}(T), \mathcal{X}(T))$ in $\mod B$, with the
torsion-free part $\mathcal{Y}(T)=\{Y \in \mod B|
\Tor^B_1(Y,T)=0\}$ and the torsion part $\mathcal{X}(T)=\{Y \in
\mod B| Y \otimes_B T=0\}$. Then, by the Brenner-Butler theorem,
the functor $\Hom_H(T,-): \mod H \to \mod B$ induces an
equivalence of $\mathcal{T}(T)$ with $\mathcal{Y}(T)$, and the
functor $\Ext^1_H(T,-): \mod H \to \mod B$ induces an equivalence
of $\mathcal{F}(T)$ with $\mathcal{X}(T)$ (see \cite{BB},
\cite{HRi1}). Further, the images $\Hom_H(T,I)$ of the
indecomposable injective modules $I$ in $\mod H$ via the functor
$\Hom_H(T,-)$ belong to one component $\mathcal{C}_T$ of
$\Gamma_B$, called the connecting component of $\Gamma_B$
determined by $T$, and form a faithful section $\Delta_T$ of
$\mathcal{C}_T$, with $\Delta_T$ the opposite valued quiver
$Q^{\op}_H$ of $Q_H$. Recall that a full connected valued
subquiver $\Sigma$ of a component $\mathcal{C}$ of $\Gamma_B$ is
called a section if $\Sigma$ has no oriented cycles, is convex in
$\mathcal{C}$, and intersects each $\tau_B$-orbit of $\mathcal{C}$
exactly once. Moreover, the section $\Sigma$ is faithful provided
the direct sum of all modules lying on $\Sigma$ is a faithful
$B$-module. The section $\Delta_T$ of the connecting component
$\mathcal{C}_T$ of $\Gamma_B$ has the distinguished property: it
connects the torsion-free part $\mathcal{Y}(T)$ with the torsion
part $\mathcal{X}(T)$, because every predecessor in $\ind B$ of a
module $\Hom_H(T,I)$ from $\Delta_T$ lies in $\mathcal{Y}(T)$ and
every successor of $\tau^-_B \Hom_H(T,I)$ in $\ind B$ lies in
$\mathcal{X}(T)$.

\begin{lem} \label{lem 3.1}
Let $H$ be an indecomposable algebra, $T$ a tilting module in
$\mod H$, and $B=\End_H(T)$ the associated tilted algebra. Then
for any injective module $I$ in $\mod H$, $M_I=\Hom_H(T,I)$ is a
module in $\mod B$ which is not the middle of a short chain.
\end{lem}
\begin{proof}
Consider the connecting component $\mathcal{C}_T$ of $\Gamma_B$
determined by $T$  and its canonical section $\Delta_T$ given by
the images of  a complete set of pairwise nonisomorphic injective
$H$-modules via the functor $\Hom_H(T, -): \mod H \rightarrow \mod
B$.  Then $M_I$ is isomorphic to a direct sum of indecomposable
modules lying on $\Delta_T$. Suppose $M_I$ is the middle of a
short chain $\xymatrix@C=13pt{X \ar[r] &  M_I \ar[r] & \tau_BX}$
in $\mod B$. Then $X$ is a predecessor in $\ind B$ of an
indecomposable module $Y$ lying on $\Delta_T$, and consequently $Y
\in \mathcal{Y}(T)$ forces $X \in \mathcal{Y}(T)$. Hence $\tau_BX$
also belongs to $\mathcal{Y}(T)$ since $\mathcal{Y}(T)$ is closed
under predecessors in $\ind B$. In particular, $\tau_BX$ does not
lie on $\Delta_T$. Then $\Hom_B(M_I,\tau_BX) \neq 0$ implies that
there is an indecomposable module $Z$ on $\Delta_T$ such that
$\tau_BX$ is a successor of $\tau_B^{-1}Z$ in $\ind B$. But then
$\tau^{-1}_BZ \in \mathcal{X}(T)$ forces $\tau_BX \in
\mathcal{X}(T)$, because $\mathcal{X}(T)$ is closed under
successors in $\ind B$. Hence the indecomposable $B$-module
$\tau_BX$ is simultaneously in $\mathcal{Y}(T)$ and
$\mathcal{X}(T)$, a contradiction. Therefore, $M_I$ is indeed a
module in $\mod B$ which is not the middle of a short chain.
\end{proof}

Recently, the authors established in \cite{JMS1} the following
characterization of tilted algebras.
\begin{thm} \label{thm jms}
An algebra $B$ is a tilted algebra if and only if $\mod B$ admits
a sincere module $M$ which is not the middle of a short chain.
\end{thm}

We exhibit now a handy criterion for an indecomposable algebra to
be a tilted algebra established independently in \cite{L2} and
\cite{S1}.
\begin{thm} \label{thm 3.3}
Let $B$ be an indecomposable algebra. Then $B$ is a tilted algebra
if and only if the Auslander-Reiten quiver $\Gamma_B$ of $B$
admits a component $\mathcal{C}$ with a faithful section  $\Delta$
such that $\Hom_B(X, \tau_BY)=0$ for all modules $X$ and $Y$ in
$\Delta$. Moreover, if this is the case and $T^{\ast}_{\Delta}$ is
the direct sum of all indecomposable modules lying on $\Delta$,
then $H_{\Delta}=\End_B(T^{\ast}_{\Delta})$ is an indecomposable
hereditary algebra, $T_{\Delta}=D(T^{\ast}_{\Delta})$ is a tilting
module in $\mod H_{\Delta}$, and the tilted algebra
$B_{\Delta}=\End_{H_{\Delta}}(T_{\Delta})$ is the basic algebra of
$B$.
\end{thm}

Let $H$ be an indecomposable hereditary algebra not of Dynkin
type, that is, the valued quiver $Q_H$ of $H$ is a Euclidean or
wild quiver. Then by a concealed algebra of type $Q_H$ we mean an
algebra $B=\End_H(T)$ for a tilting module $T$ in
$\add(\mathcal{P}(H))$ (equivalently, in $\add(\mathcal{Q}(H))$).
If $Q_H$ is a Euclidean quiver, $B$ is said to be a tame concealed
algebra. Similarly, if $Q_H$ is a wild quiver, $B$ is said to be a
wild concealed algebra. Recall that the Auslander-Reiten quiver
$\Gamma_B$ of a concealed algebra $B$ is of the form:
\[\Gamma_B=\mathcal{P}(B) \vee \mathcal{R}(B) \vee \mathcal{Q}(B),\]
where $\mathcal{P}(B) $ is a preprojective component  containing
all indecomposable projective $B$-modules, $\mathcal{Q}(B)$ is a
preinjective component  containing all indecomposable injective
$B$-modules and $\mathcal{R}(B)$ is either an infinite family of
stable tubes separating $\mathcal{P}(B)$ from $\mathcal{Q}(B)$ or
an infinite family of components of type $\mathbb{ZA}_{\infty}$.

\begin{prop}\label{prop 3.4}
Let $B$ be a wild concealed algebra, $\mathcal{C}$ a regular
component of $\Gamma_B$, $M$ a module in $\mathcal{P}(B)$ and $N$
a module in $\mathcal{Q}(B)$. Then the following statements hold:
\begin{enumerate}
\renewcommand{\labelenumi}{\rm(\roman{enumi})}
\item $\Hom_B(M,X)\neq 0$ for all but finitely many modules $X$ in
$\mathcal{C}$;
\item $\Hom_B(X,N)\neq 0$ for all but finitely many modules $X$ in $\mathcal{C}$.
\end{enumerate}
In particular, all but finitely many modules in $\mathcal{C}$ are
sincere.
\end{prop}

\begin{proof} (i) Let $H$ be a wild hereditary algebra and $T$ a tilting module in $\add(\mathcal{P}(H))$ such that $B=\End_H(T)$. Recall that the
functor $\Hom_H(T,-): \mod H \rightarrow \mod B$ induces an
equivalence of the torsion part $\mathcal{T}(T)$ of $\mod H$ and
the torsion-free part $\mathcal{Y}(T)$ of $\mod B$. Moreover, we
have the following facts:
\begin{itemize}
\item[{(a)}] the images under the functor $\Hom_H(T,-)$ of the regular components from $\mathcal{R}(H)$ form the family $\mathcal{R}(B)$ of all regular components
of $\Gamma_{B}$;
\item[{(b)}] the images under the functor $\Hom_H(T,-)$ of all indecomposable modules in $\mathcal{P}(H)\cap\mathcal{T}(T)$ form the unique preprojective
component $\mathcal{P}(B)$ of $\Gamma_B$.
\end{itemize}

\noindent Since $\mathcal{C}$ is in $\mathcal{R}(B)$, there exists
a component $\mathcal{D}$ in $\mathcal{R}(H)$ such that
$\mathcal{C}=\Hom_H(T,\mathcal{D})$. We note that $\mathcal{C}$
and $\mathcal{D}$ are of the form ${\Bbb Z}{\Bbb A}_{\infty}$. It
follows from \cite{Bae} (see also \cite[Corollary XVIII.2.4]{SS2})
that all but finitely many modules in $\mathcal{D}$ are sincere
$H$-modules. We may choose an indecomposable module $U$ in
$\mathcal{P}(H)\cap\mathcal{T}(T)$ such that $M=\Hom_H(T,U)$.
Further, there exists an indecomposable projective module $P$ in
$\mathcal{P}(H)$ such that $U=\tau_H^{-m}P$ for some integer
$m\geq 0$. Take now an indecomposable module $Z$ in $\mathcal{D}$.
Then we obtain isomorphisms of $R$-modules
\[ \Hom_H(U,Z)\cong\Hom_H(\tau_H^{-m}P,Z)\cong\Hom_H(P,\tau_H^{m}Z), \]
because $H$ is hereditary (see \cite[Corollary IV.2.15]{ASS}).
Since $\Hom_H(P,R)\neq 0$ for all but finitely many modules $R$ in
$\mathcal{D}$, we conclude that $\Hom_H(U,Z)\neq 0$ for all but
finitely many modules $Z$ in $\mathcal{D}$. Applying now the
equivalence of categories $\Hom_H(T,-): \mathcal{T}(T)\to
\mathcal{Y}(T)$ and the equalities
$\mathcal{P}(B)=\Hom_H(T,\mathcal{P}(H)\cap\mathcal{T}(T))$,
$\mathcal{C}=\Hom_H(T,\mathcal{D})$, and $M=\Hom_H(T,U)$, we
obtain that $\Hom_B(M,X)\neq 0$ for all but finitely many modules
$X$ in $\mathcal{C}$.

(ii) We note that the preinjective component $\mathcal{Q}(B)$ is
the connecting component $\mathcal{C}_T$ of $\Gamma_B$ determined
by $T$, and is obtained by gluing the image
$\Hom_H(T,\mathcal{Q}(H))$ of the preinjective component
$\mathcal{Q}(H)$ of $\Gamma_H$ with a finite part consisting of
all indecomposable modules of the torsion part
$\mathcal{X}(T)=\Ext_H^1(T,\mathcal{F}(T))$ of $\mod B$ (see
\cite[Theorem VIII.4.5]{ASS}). But the wild concealed algebra $B$
is also of the form $B=\End_{H^*}(T^*)$, where $H^*$ is a wild
hereditary algebra and $T^*$ is a tilting module in
$\add(\mathcal{Q}(H^*))$. Then the functor $\Ext_{H^*}^1(T^*,-):
\mod H^*\to\mod B$ induces an equivalence of the torsion-free part
$\mathcal{F}(T^*)$ of $\mod H^*$ and the torsion part
$\mathcal{X}(T^*)$ of $\mod B$. Moreover, we have the following
facts:
\begin{itemize}
\item[{(a$^*$)}] the images under the functor $\Ext_{H^*}^1(T^*,-)$ of the regular components from $\mathcal{R}(H^*)$ form the family $\mathcal{R}(B)$ of all regular
components of $\Gamma_B$;
\item[{(b$^*$)}] the images under the functor $\Ext_{H^*}^1(T^*,-)$ of all indecomposable modules in $\mathcal{Q}(H)\cap\mathcal{F}(T)$ form the
unique preinjective component $\mathcal{Q}(B)$ of $\Gamma_B$.
\end{itemize}
In particular, we have that
$\mathcal{C}=\Ext_{H^*}^1(T^*,\mathcal{D}^*)$ for a component
$\mathcal{D}^*$ in $\mathcal{R}(H^*)$. We then conclude that
$\Hom_B(X,N)\neq 0$ for all but finitely many modules $X$ in
$\mathcal{C}$, applying arguments dual to those used in the proof
of (i).

The fact that all but finitely many modules in $\mathcal{C}$ are
sincere follows from (i) (equivalently (ii)), because
$\mathcal{P}(B)$ contains all indecomposable projective
$B$-modules and $\mathcal{Q}(B)$ contains all indecomposable
injective $B$-modules.
\end{proof}

A prominent role in our considerations will be played by the
following consequence of a result of D. Baer \cite{Bae} (see
\cite[Theorem XVIII.5.2]{SS2}).
\begin{thm} \label{thm 3.5}
Let $B$ be a wild concealed algebra, and $M,N$ indecomposable
$B$-modules lying in regular components of $\Gamma_B$. Then there
exists a positive integer $m_0$ such that $\Hom_B(M,
\tau_B^mN)\neq 0$ for all integers $m \geqslant m_0$.
\end{thm}

\begin{lem} \label{lem 3.6}
Let $B$ be a wild concealed algebra and $\mathcal{C}$ a regular
component of $\Gamma_B$. Then any indecomposable module $N$ in
$\mathcal{C}$ is the middle of a short chain in $\mod B$.
\end{lem}
\begin{proof}
Suppose $N$ is an indecomposable module in $\mathcal{C}$.
Obviously $\mathcal{C}$ is of the form $\mathbb{ZA}_{\infty}$.
Applying Theorem \ref{thm 3.5}, we conclude that there is a
positive integer $m_0$ such that $\Hom_B(N, \tau^m_BN) \neq 0$ for
all integers $m \geq m_0$. Then we may take an indecomposable
module $X$ in $\mathcal{C}$ such that there are a sectional path
$\Omega$ from $X$ to $N$ and a sectional path $\Sigma$ from
$\tau^m_BN$ to $\tau_BX$ for some integer $m \geq m_0$. Observe
that all irreducible homomorphisms corresponding to arrows of
$\Sigma$ are monomorphisms whereas all irreducible  homomorphisms
corresponding to arrows of $\Omega$ are epimorphisms. Hence there
are a monomorphism $f: \tau^m_BN \rightarrow \tau_BX$ and an
epimorphism $g: X \rightarrow N$. Since $\Hom_B(N, \tau_B^mN) \neq
0$, we conclude that $\Hom_B(N, \tau_BX) \neq0$. Therefore, we
obtain a short chain  $\xymatrix@C=13pt{X \ar[r] &  N \ar[r] &
\tau_BX}$.
\end{proof}


\section{\normalsize Proofs of Theorem \ref{thm 1.1} and Corollary \ref{cor 1.2}}

Let $A$ be an algebra and $M$ a module in $\mod A$ which is not
the middle of a short chain. By $\ann_A(M)$ we shall denote the
annihilator of $M$ in $A$, that is, the ideal $\{a \in A| Ma=0\}$.
Then $M$ is  a sincere module over the algebra $B=A/ \ann_A(M)$.
Moreover, by \cite[Proposition 2.3]{RSS}, $M$ is not the middle of
a short chain in $\mod B$, since $M$  is not the middle of a short
chain in $\mod A$. Let $B=B_1\times\ldots\times B_m$ be a
decomposition of $B$ into a product of indecomposable algebras and
$M=M_1\oplus\ldots\oplus M_m$ the associated decomposition of $M$
in $\mod B$ with $M_i$ a module in $\mod B_i$ for any $i\in \{1,
\ldots, m\}$. Observe that, for each $i\in \{1, \ldots, m\}$,
$B_i=A/ \ann_A(M_i)$, $M_i$ is a sincere $B_i$-module which is not
the middle of a short chain in $\mod B_i$, and hence $B_i$
is a tilted algebra, by Theorem \ref{thm jms}. Therefore, we may assume that $B$ is an indecomposable algebra. \\

We will start our considerations by showing that for a tilted
algebra $B$ and a sincere $B$-module $M$ which is not the middle
of a short chain, all indecomposable direct summands of $M$ belong
to the same component, which is in fact a connecting component of
$\Gamma_B$. According to a result of C. M. Ringel \cite[p.46]{Ri2}
$\Gamma_B$ admits at most two components containing sincere
sections (slices), and exactly two if and only if $B$ is a
concealed algebra. We shall discuss this case in the following
proposition.

\begin{prop}\label{prop 4.1}
Let $B$ be a concealed algebra and $M$ a sincere $B$-module which is not the middle of a short chain. Then $M \in
\add (\mathcal{C})$ for a connecting component $\mathcal{C}$ of $\Gamma_B$.
\end{prop}

\begin{proof}

Observe that $M$ has no indecomposable direct summands in $\mathcal{R}(B)$, by Lemmas \ref{lem 2.4} and \ref{lem 3.6}. Hence we may assume that $M=M_P \oplus M_Q$, where $M_P$ is a direct summand of $M$ contained in $\add(\mathcal{P}(B))$,
whereas $M_Q$ is a direct summand of $M$ which belongs to $\add(\mathcal{Q}(B))$.
We claim that $M_P=0$ or $M_Q=0$. Suppose $M_P \neq 0$ and $M_Q \neq 0$.
 Let $M'$  be an indecomposable
direct summand of $M_P$ and $M''$ an indecomposable direct summand of $M_Q$.

Consider the case when $B$ is a concealed algebra of Euclidean type, that is, $\mathcal{R}(B)$ is a family of stable tubes.
Then it follows from Lemma \ref{lem 2.3}
that there is a module $Z$ in $\mathcal{R}(B)$ such that
$\Hom_B(M', \tau_BZ)\neq 0$ and $\Hom_B(Z, M'')\neq 0$. This contradicts the assumption that $M$ is not the middle
of a short chain. Hence $M_P=0$ or $M_Q=0$.

Assume now that $B$  is a wild concealed algebra. Fix a regular component $\mathcal{D}$ of $\Gamma_{B}$. Invoking
Proposition \ref{prop 3.4} we conclude  that there exists a module $X \in \mathcal{D}$  such that  $\Hom_B(M',\tau_BX)\neq 0$
and $\Hom_B(X, M'') \neq 0$. Thus $M$ is the middle of a short chain $X \rightarrow M \rightarrow \tau_BX$ in $\mod B$.
Hence, we get that $M_P=0$ or $M_Q=0$.

Therefore, we obtain that $M$ belongs to $\add(\mathcal{C})$ for a connecting component $\mathcal{C} =\mathcal{P}(B)$
or $\mathcal{C}=\mathcal{Q}(B)$.
\end{proof}

We shall now be concerned with the situation of exactly one
connecting component in the Auslander-Reiten quiver of a tilted
algebra. Let $H$  be an indecomposable hereditary algebra of
infinite representation type, $T$  a tilting module in $\mod H$
and $B=\End_H(T)$ the associated tilted algebra. By
$\mathcal{C}_T$ we denote the connecting component in $\Gamma_B$
determined by $T$. We keep these notations to formulate and prove
the following statement.

\begin{prop}\label{prop 4.2}
Let $B=\End_H(T)$ be an indecomposable tilted algebra which is not
concealed. If $M$ is a sincere $B$-module  which is not  the
middle of a short chain in $\mod B$, then $M \in
\add(\mathcal{C}_T)$.
\end{prop}

\begin{proof}
We start with the general view on the module category $\mod B$ due
to results established in \cite{K1}, \cite{K2}, \cite{K3},
\cite{L1}, \cite{St}. Let $\Delta= \Delta_T$ be the canonical
section of the connecting component $\mathcal{C}_T$ of $\Gamma_B$
determined by $T$. Hence, $\Delta= Q^{\op}$ for $Q=Q_H$. Then
$\mathcal{C}_T$ admits a finite (possibly empty) family of
pairwise disjoint full translation (valued) subquivers
\[\mathcal{D}^{(l)}_1, ..., \mathcal{D}^{(l)}_m, \mathcal{D}^{(r)}_1, ..., \mathcal{D}^{(r)}_n\]
such that  the following statements hold:
\begin{itemize}
\item[(a)] for each $i \in \{1,...,m\}$, there is an isomorphism of translation quivers $\mathcal{D}^{(l)}_i
\cong \mathbb{N} \Delta^{(l)}_i$, where $\Delta^{(l)}_i$ is a
connected full valued subquiver of $\Delta$, and
$\mathcal{D}^{(l)}_i$ is closed under predecessors in
$\mathcal{C}_T$;
\item[(b)] for each $j \in \{1,...,n\}$, there is an isomorphism of translation quivers $\mathcal{D}^{(r)}_j
\cong (-\mathbb{N}) \Delta^{(r)}_j$, where $\Delta^{(r)}_j$ is a
connected full valued subquiver of $\Delta$, and
$\mathcal{D}^{(r)}_j$ is closed under successors in
$\mathcal{C}_T$;
\item[(c)] all but finitely many indecomposable modules of $\mathcal{C}_T$ lie in
\[\mathcal{D}^{(l)}_1 \cup ...\cup \mathcal{D}^{(l)}_m \cup \mathcal{D}^{(r)}_1 \cup ...\cup \mathcal{D}^{(r)}_n;\]
\item[(d)] for each $i \in \{1,...,m\}$, there exists a tilted algebra $B^{(l)}_i=\End_{H^{(l)}_i}(T^{(l)}_i)$,
where $H^{(l)}_i$ is a hereditary algebra of type $(\Delta^{(l)}_i)^{\op}$ and $T^{(l)}_i$ is a
tilting $H^{(l)}_i$-module without preinjective indecomposable direct summands such that
\begin{itemize}
\item[$\bullet$] $B^{(l)}_i$ is a quotient algebra of $B$, and hence there is a fully faithful embedding
$\mod B^{(l)}_i \hookrightarrow \mod B$,
\item[$\bullet$] $\mathcal{D}^{(l)}_i$ coincides with the torsion-free part $\mathcal{Y}(T^{(l)}_i) \cap
\mathcal{C}_{T^{(l)}_i}$ of the connecting component
$\mathcal{C}_{T^{(l)}_i}$ of $\Gamma_{B^{(l)}_i}$ determined by
$T^{(l)}_i$;
\end{itemize}
\item[(e)]for each $j \in \{1,...,n\}$, there exists a tilted algebra $B^{(r)}_j=\End_{H^{(r)}_j}(T^{(r)}_j)$,
where $H^{(r)}_j$ is a hereditary algebra of type $(\Delta^{(r)}_j)^{\op}$ and $T^{(r)}_j$ is a
tilting $H^{(r)}_j$-module without preprojective indecomposable direct summands such that
\begin{itemize}
\item[$\bullet$] $B^{(r)}_j$ is a quotient algebra of $B$, and hence there is a fully faithful embedding
$\mod B^{(r)}_j \hookrightarrow \mod B$,
\item[$\bullet$] $\mathcal{D}^{(r)}_j$ coincides with the torsion part $\mathcal{X}(T^{(r)}_j) \cap
\mathcal{C}_{T^{(r)}_j}$ of the connecting component
$\mathcal{C}_{T^{(r)}_j}$ of $\Gamma_{B^{(r)}_j}$ determined by
$T^{(r)}_j$;
\end{itemize}
\item[(f)] $\mathcal{Y}(T)=\add (\mathcal{Y}(T^{(l)}_1) \cup ... \cup \mathcal{Y}(T^{(l)}_m)\cup (\mathcal{Y}(T)
\cap \mathcal{C}_T))$;
\item[(g)] $\mathcal{X}(T)=\add ((\mathcal{X}(T) \cap \mathcal{C}_T)\cup \mathcal{X}(T^{(r)}_1) \cup ... \cup
\mathcal{X}(T^{(r)}_n))$;
\item[(h)] the Auslander-Reiten quiver $\Gamma_B$ has the disjoint union form
\[\Gamma_B = (\bigcup_{i=1}^m \mathcal{Y}\Gamma_{B^{(l)}_i}) \cup \mathcal{C}_T \cup  (\bigcup_{j=1}^n
\mathcal{X}\Gamma_{B^{(r)}_j}),\]
where
\begin{itemize}
\item[$\bullet$] for each $i \in \{1,...,m\}$, $\mathcal{Y}\Gamma_{B^{(l)}_i}$ is the union of all components
of $\Gamma_{B^{(l)}_i}$ contained entirely in $\mathcal{Y}(T^{(l)}_i)$,
\item[$\bullet$] for each $j \in \{1,...,n\}$, $\mathcal{X}\Gamma_{B^{(r)}_j}$ is the union of all components
of $\Gamma_{B^{(r)}_j}$ contained entirely in
$\mathcal{X}(T^{(r)}_j)$.
\end{itemize}
\end{itemize}
Moreover, we have the following description of the components of
$\Gamma_B$ contained in the parts $\mathcal{Y}\Gamma_{B^{(l)}_i}$
and $\mathcal{X}\Gamma_{B^{(r)}_j}$.
\begin{itemize}
\item[(1)] If $\Delta^{(l)}_i$ is a Euclidean quiver, then $\mathcal{Y}\Gamma_{B^{(l)}_i}$ consists
of a unique preprojective component $\mathcal{P}({B^{(l)}_i})$ of
$\Gamma_{{B^{(l)}_i}}$ and an infinite family
$\mathcal{T}^{B^{(l)}_i}$ of pairwise orthogonal generalized
standard ray tubes. Further, $\mathcal{P}({B^{(l)}_i})$ coincides
with the preprojective component $\mathcal{P}({C^{(l)}_i})$ of a
tame concealed quotient algebra $C^{(l)}_i$ of $B^{(l)}_i$.
\item[(2)] If $\Delta^{(l)}_i$ is a wild quiver, then $\mathcal{Y}\Gamma_{B^{(l)}_i}$ consists of
a unique preprojective component $\mathcal{P}(B^{(l)}_i)$ of
$\Gamma_{B^{(l)}_i}$ and an infinite family of components obtained
from the components of the form $\mathbb{ZA}_{\infty}$ by a finite
number (possibly empty) of ray insertions. Further,
$\mathcal{P}(B^{(l)}_i)$ coincides with the preprojective
component $\mathcal{P}(C^{(l)}_i)$ of a wild concealed quotient
algebra $C^{(l)}_i$ of $B^{(l)}_i$.
\item[(3)] If $\Delta^{(r)}_j$ is a Euclidean quiver, then $\mathcal{X}\Gamma_{B^{(r)}_j}$ consists
of a unique preinjective component $\mathcal{Q}(B^{(r)}_j)$ of
$\Gamma_{B^{(r)}_j}$ and an infinite family
$\mathcal{T}^{B^{(r)}_j}$ of pairwise orthogonal generalized
standard coray tubes. Further, $\mathcal{Q}(B^{(r)}_j)$ coincides
with the preinjective component $\mathcal{Q}(C^{(r)}_j)$ of a tame
concealed quotient algebra $C^{(r)}_j$ of $B^{(r)}_j$.
\item[(4)] If $\Delta^{(r)}_j$ is a wild quiver, then $\mathcal{X}\Gamma_{B^{(r)}_j}$ consists of
a unique preinjective component $\mathcal{Q}(B^{(r)}_j)$ of $\Gamma_{B^{(r)}_j}$ and an infinite family of
components obtained from the components of the form $\mathbb{ZA}_{\infty}$ by a finite number (possibly
empty) of coray insertions. Further, $\mathcal{Q}(B^{(r)}_j)$ coincides with the preinjective component
$\mathcal{Q}(C^{(r)}_j)$ of a wild concealed quotient algebra $C^{(r)}_j$ of $B^{(r)}_j$.
\end{itemize}

Observe that each indecomposable $B$-module belongs either to
$\mathcal{Y}(T)$ or $\mathcal{X}(T)$ ($T$ is a splitting tilting
module). Let $M'$  be an indecomposable direct summand of $M$,
which is contained in $\mathcal{Y}(T)$. We claim that then $M'$
belongs to $\mathcal{Y}(T) \cap \mathcal{C}_T$. Conversely, assume
that  $M' \in \mathcal{Y}(T)\backslash \mathcal{C}_T$. Then there
exists $i \in \{1,..., m\}$  such that $M' \in
\mathcal{Y}\Gamma_{B^{(l)}_i} \backslash \mathcal{C}_T$,
equivalently $M' \in \mathcal{Y}\Gamma_{B^{(l)}_i} \backslash
\mathcal{C}_{T^{(l)}_i}$. Without loss of generality we may assume
that $i=1$. Since $T^{(l)}_1$ does not contain indecomposable
preinjective direct summands,
we may distinguish two cases.\\

Assume first that $T^{(l)}_1$ contains an indecomposable direct
summand from $\mathcal{R}(H^{(l)}_1)$. This implies that there is
a projective module $P$ in $\mathcal{Y}\Gamma_{B^{(l)}_1}$ which
does not belong to $\mathcal{P}(B^{(l)}_1)$. If $B_{1}^{(l)}$ is a
tilted algebra of Euclidean type, then $P$ is a module from some
ray tube $\mathcal{T}$. Then, according to Lemma \ref{lem 2.2},
$\Hom_{B^{(l)}_1}(M',X)=0$ for any $X \in \mathcal{T}$, which
leads to conclusion that $M' \notin \mathcal{P}(B^{(l)}_1)$,
because $\mathcal{T}$ belongs to the family
$\mathcal{T}^{B^{(l)}_i}$ of ray tubes separating
$\mathcal{P}(B^{(l)}_i)$ from the preinjective component
$\mathcal{Q}(B^{(l)}_i)$ of $\Gamma_{B^{(l)}_i}$. Since $M'$ does
not belong to an infinite family of ray tubes (Lemmas \ref{lem
2.2} and \ref{lem 2.4}), by (1) we conclude that $M'=0$, a
contradiction. This shows that any indecomposable direct summand
of $M$ from $\mathcal{Y}(T)$ is contained in $\mathcal{Y}(T) \cap
\mathcal{C}_T$. If $B^{(l)}_1$ is a tilted algebra of wild type,
then $P$ belongs to a component obtained from a component of type
$\mathbb{ZA}_{\infty}$, say $\mathcal{D}$, by a positive number of
ray insertions. Then there is a left cone $(\rightarrow N)$ in
$\mathcal{D}$ which consists only of $C^{(l)}_1$-modules
\cite[Theorem 1]{K2}. Moreover,
$\tau_{C^{(l)}_1}V=\tau_{B^{(l)}_1}V$ for any module $V \in
(\rightarrow N)$ and there is an indecomposable module $Y \in
\mathcal{R}(H^{(l)}_1)$ such that $N=\Hom_{H^{(l)}_1}(T^{(l)}_1,
Y)$. Because $M' \in \mathcal{Y}\Gamma_{B^{(l)}_1} \backslash
\mathcal{C}_{T^{(l)}_1}$, we have also that $M' =
\Hom_{H^{(l)}_1}(T^{(l)}_1, X)$ for some $X \in
\mathcal{P}(H^{(l)}_1) \cup \mathcal{R}(H^{(l)}_1)$.

Suppose that $X \in \mathcal{R}(H^{(l)}_1)$. Invoking Theorem
\ref{thm 3.5}, we have that there exists a positive integer $t$
such that $\Hom_{H^{(l)}_1}(X, \tau^p_{H^{(l)}_1}Y)\neq 0$ for all
integers $p \geq t$. This implies that also
$\Hom_{B^{(l)}_1}(M',\tau_{B^{(l)}_1}^pN) \neq 0$ for all integers
$p \geq t$. Moreover, if $X \in \mathcal{P}(H^{(l)}_1)$, then
$M'=\Hom_{H^{(l)}_1}(T^{(l)}_1,X)$ belongs to
$\mathcal{P}(B^{(l)}_1)$, which is equal to
$\mathcal{P}(C^{(l)}_1)$. From Proposition \ref{prop 3.4} we
obtain now that there exists a positive integer $t$ such that
$\Hom_{B^{(l)}_1} (M', \tau_{B^{(l)}_1}^pN) \neq 0$ for all
integers $p \geq t$. Thus we have, independently on the position
of $X$ in $\mathcal{P}(H^{(l)}_1) \cup \mathcal{R}(H^{(l)}_1)$, a
nonzero homomorphism $g: M' \rightarrow \tau^p_{B^{(l)}_1}N$, for
any integer  $p \geq t$. Observe also that $M$ is a faithful
$B$-module because $M$ is sincere and not the middle of a short
chain (see \cite[Corollary 3.2]{RSS}). Hence there is a
monomorphism $B_B\to M^r$, for some positive integer $r$, so we
have a monomorphism $P\to M^r$, because $P$ is a direct summand of
$B_B$. Further, since $\mathcal{D}$ contains a finite number of
projective modules, we may assume, without loss of generality,
that $P$ is the one whose radical has an indecomposable direct
summand $L$ such that $\tau^s_{B^{(l)}_1}L \neq 0$ for any integer
$s \geq 1$. Consider the infinite sectional path $\Sigma$ in
$\mathcal{D}$ which terminates at $L$. Then there exists an
integer $p \geq t$ such that the infinite sectional path $\Omega$
which starts at $\tau_{B^{(l)}_1}^p N$ contains a module
$\tau_{B^{(l)}_1}Z$ with $Z$ lying on $\Sigma$. Then,
$\Hom_{B^{(l)}_1}(Z, L) \neq 0$, by Lemma \ref{lem 2.5}, and hence
$\Hom_{B^{(l)}_1}(Z, M) \neq 0$, since there are monomorphisms $L
\rightarrow P$ and $P \rightarrow M^r$ for some integer $r \geq
1$. Similarly, we obtain $\Hom_{B^{(l)}_1}(M', \tau_{B^{(l)}_1}Z)
\neq 0$, because there are a nonzero homomorphism $g: M'\to
\tau^p_{B^{(l)}_1}N$ and a monomorphism from $\tau_{B^{(l)}_1}^pN$
to $\tau_{B^{(l)}_1}Z$ being a composition of irreducible
monomorphisms. Finally, we get a short chain $Z \rightarrow M
\rightarrow \tau_{B^{(l)}_1}Z$ in $\mod B$, which contradicts
the assumption imposed on $M$. \\

Assume now that $T^{(l)}_1$ belongs to
$\add(\mathcal{P}(H^{(l)}_1))$. Then $B^{(l)}_1$ is a concealed
algebra and $B \neq B^{(l)}_1$, since $B$ is not concealed by the
assumption. Therefore, since $B$ is indecomposable, there exists a
module $R \in \mathcal{Q}(B^{(l)}_1)$, more precisely, a module
$R\in  \mathcal{Q}(B^{(l)}_1) \cap \mathcal{C}_{T}$ such that $W$
is a direct summand of $\rad P$ of some projective $B$-module $P$.
Moreover, by Lemmas \ref{lem 2.4} and \ref{lem 3.6}, we obtain
that $M' \in \mathcal{Y}\Gamma_{B^{(l)}_1}\backslash
\mathcal{C}_{T^{(l)}_1}$ implies $M' \in \mathcal{P}(B^{(l)}_1)$.
We claim that there exists $Z \in \mathcal{R}(B^{(l)}_1)$ such
that $\Hom_{B^{(l)}_1}(Z, W) \neq 0$ and $\Hom_{B^{(l)}_1}(M',
\tau_{B^{(l)}_1}Z) \neq 0$.

If $B^{(l)}_1$ is of Euclidean type then the claims follows from
Lemma \ref{lem 2.3}. Suppose now that $B^{(l)}_1$ is of wild type.
Let $\mathcal{D}$ be a fixed component in
$\mathcal{R}(B^{(l)}_1)$. From Proposition \ref{prop 3.4} we know
that there are nonzero homomorphisms $M' \rightarrow U$ for almost
all $U \in \mathcal{D}$. Also by this proposition, for almost all
$U \in \mathcal{D}$, there is a nonzero homomorphism $U
\rightarrow W$. Thus, we conclude that there exists a regular
$B^{(l)}_1$-module $Z$ such that $\Hom_{B^{(l)}_1}(Z, W) \neq 0$
and $\Hom_{B^{(l)}_1}(M', \tau_{B^{(l)}_1}Z) \neq 0$. Combining
now a nonzero homomorphism from $Z$ to $W$ with the composition of
monomorphisms $W \rightarrow P$ and $P \rightarrow M^r$, for some
integer $r\geq 1$, we obtain that $\Hom_{B^{(l)}_1}(Z, M) \neq 0$.
Consequently,
there is  a short chain $Z \rightarrow M \rightarrow \tau_{B^{(l)}_1}Z$ in $\mod B$, a contradiction.\\

We use dual arguments to show that any indecomposable direct summand $M''$ of $M$, which is contained in
$\mathcal{X}(T)$, belongs in fact to $\mathcal{X}(T) \cap \mathcal{C}_T$.
\end{proof}

\vspace{0,2cm}
\begin{prop}\label{prop 4.3}
Let $B=\End_H(T)$ be an indecomposable tilted algebra,
$\mathcal{C}_T$ the connecting component of $\Gamma_B$ determined
by $T$, and $M$ a sincere module in $\add(\mathcal{C}_T)$ which is
not the middle of a short chain. Then there is a section $\Delta$
in $\mathcal{C}_T$ such that every indecomposable direct summand
of $M$ belongs to $\Delta$.
\end{prop}

\begin{proof}
We divide the proof into several steps.

(1) Let $M'$ be an indecomposable direct summand of $M$ and $R$ be
an immediate predecessor of some projective module $P$ in
$\mathcal{C}_T$ (if $\mathcal{C}_T$ contains a projective module).
We prove that, if $M'$ is a predecessor of $R$ in $\mathcal{C}_T$,
then $M'$ belongs to $\mathcal{S}_R$. Assume that $M'$ is a
predecessor of $R$ in $\mathcal{C}_T$ and $M'$ does not belong to
$\mathcal{S}_R$. Since $R$ has no injective nonsectional
predecessors in $\mathcal{C}_T$, we have from Proposition
\ref{prop 2.6} (i) that $\Hom_B(M',\tau_BU)\neq 0$ for some module
$U\in\mathcal{S}_R$. Moreover, $\Hom_B(U,R)\neq 0$, because there
is a sectional path from $U$ to $R$ in $\mathcal{C}_T$. Since $M$
is faithful, there is a monomorphism $B_B\to M^r$, for some
positive integer $r$, so we have a monomorphism $P\to M^r$,
because $P$ is a direct summand of $B_B$. Combining now a nonzero
homomorphism from $U$ to $R$ with the composition of monomorphisms
$R\to P$ and $P\to M^r$, we obtain $\Hom_B(U,M^r)\neq 0$, and
hence $\Hom_B(U,M)\neq 0$. Summing up, we have in $\mod B$ a short
chain $U\to M\to \tau_BU$, a contradiction.

Dually, using Proposition \ref{prop 2.6}(ii) we show that, if an
indecomposable direct summand $M''$ of $M$ is a successor of an
immediate successor $J$ of some injective module $I$ in
$\mathcal{C}_T$,
then $M''$ belongs to $\mathcal{S}^*_J$. \\


(2) Let $M'$ and $M''$ be nonisomorphic indecomposable direct
summands of $M$ such that $M'$ is a predecessor of $M''$ in
$\mathcal{C}_T$. We show that every  path from $M'$ to $M''$ in
$\mathcal{C}_T$ is sectional. Assume for the contrary that there
exists a nonsectional path from $M'$ to $M''$ in $\mathcal{C}_T$.
For each nonsectional path $\sigma$ in $\mathcal{C}_T$ from $M'$
to $M''$, we denote by $n(\sigma)$ the length of the maximal
sectional subpath of $\sigma$ ending in $M''$. Among the
nonsectional paths in $\mathcal{C}_T$ from $M'$ to $M''$ we may
choose a path $\gamma$ with maximal $n(\gamma)$. Let $Y_0 \to Y_1
\to\cdots\to Y_{n-1} \to Y_n=M''$ be the maximal sectional subpath
of $\gamma$ ending in $M''$. Observe that then $\gamma$ admits a
subpath of the form $\tau_BY_1 \to Y_0 \to Y_1$, and so $Y_1$ is
not projective.

We show first that there is no sectional path in $\mathcal{C}_T$
from $M'$ to $Y_0$. Note that there is no sectional path in
$\mathcal{C}_T$ from $M'$ to $\tau_BY_1$. Indeed, otherwise
$\Hom_B(M',\tau_BY_1)\neq 0$ and clearly $\Hom_B(Y_1,M'')\neq 0$,
since there is a sectional path from $Y_1$ to $Y_n=M''$, and
consequently $M$ is the middle of a short chain $Y_1 \to M \to
\tau_BY_1$, a contradiction. Moreover, applying (1), we conclude
that $Y_0$ and $\tau_BY_1$ are not projective. We claim that
$\tau_BY_1$ is a unique immediate predecessor of $Y_0$ in
$\mathcal{C}_T$. Suppose that $Y_0$ admits an immediate
predecessor $L$ in $\mathcal{C}_T$ different from $\tau_BY_1$.
Since there is no sectional path in $\mathcal{C}_T$ from $M'$ to
$\tau_BY_1$, we conclude that $\gamma$ contains a subpath of the
form
\[ M'=N_0\to N_1\to\cdots\to N_s=\tau_BZ_1\to Z_0\to Z_1\to\cdots\to Z_{t-1}\to Z_t=\tau_BY_1. \]
Assume first that all modules $Z_2, \ldots, Z_{t-1}$ are nonprojective. Then there is in $\mathcal{C}_T$ a nonsectional path $\beta$ from $M'$ to $M''$
of the form
\[ M'=N_0\to N_1\to\cdots\to \tau_BZ_1\to \tau_BZ_2 \to\cdots\to \tau_BZ_t\to \tau_BY_0\to L\to \]
\[\to Y_0\to Y_1\to\cdots\to Y_n=M'' \]
with $n(\beta)=n(\gamma)+1$, a contradiction with the choice of $\gamma$. Assume now that one of the modules $Z_2, \ldots, Z_{t-1}$ is projective.
Choose $k\in\{2,\ldots,t-1\}$ such that $Z_k$ is projective but $Z_{k+1},\ldots, Z_{t-1}, Z_t$ are nonprojective. Then $\tau_BZ_{k+1}$ is an immediate
predecessor of $Z_k$ in $\mathcal{C}_T$ and hence, applying (1), we infer that there is a sectional path in $\mathcal{C}_T$ from $M'$ to $\tau_BZ_{k+1}$.
We obtain then a nonsectional path $\alpha$ in $\mathcal{C}_T$ of the form
\[ M'\to\cdots\to \tau_BZ_{k+1}\to\cdots\to \tau_BZ_t\to \tau_BY_0\to L\to Y_0\to Y_1\to\cdots\to Y_n=M'' \]
with $n(\alpha)=n(\gamma)+1$, again a contradiction with the
choice of $\gamma$. Summing up, we proved that $Y_0$, $Y_1$ are
nonprojective and $\tau_BY_1$ is a unique immediate predecessor of
$Y_0$ in $\mathcal{C}_T$. Hence every sectional path in
$\mathcal{C}_T$ from $M'$ to $Y_0$ passes through $\tau_BY_1$.
This proves our claim, because there is no sectional path in
$\mathcal{C}_T$ from $M'$ to $\tau_BY_1$.

Observe that $\Hom_B(Y_0,M)\neq 0$, since we have a sectional path
in $\mathcal{C}_T$ from $Y_0$ to the direct summand $M''$ of $M$.
Denote by $f$ a nonzero homomorphism in $\mod B$ from $Y_0$ to $M$
and consider a projective cover $g: P_B(Y_0) \to Y_0$ of $Y_0$ in
$\mod B$. Then $fg\neq 0$ and hence there exist an indecomposable
projective $B$-module $P$ and nonzero homomorphism $h: P\to Y_0$
such that $fh\neq 0$. Applying (1) and Proposition \ref{prop 2.6}
(ii), we conclude that $h$ factorizes through a module in
$\add(\tau^-_B\mathcal{S}^*_{M'})$. Then there exists a module $U$
in $\mathcal{S}^*_{M'}$ and a nonzero homomorphism $j: \tau^-_BU
\to Y_0$ such that $fj\neq 0$. Moreover, $\Hom_B(M',U)\neq 0$
because there is a sectional path from $M'$ to $U$ in
$\mathcal{C}_T$. Therefore, $M$ is the middle of a short chain
$\tau^-_BU \to M\to U$, with $U=\tau_B (\tau^-_BU)$, a contradiction. \\

(3) Let $M'$ be an indecomposable direct summand of $M$ which is a
predecessor of an indecomposable projective module $P$ in
$\mathcal{C}_T$. Then every path in $\mathcal{C}_T$ from $M'$ to
$P$ is sectional. Indeed, since $M$ is a faithful module in $\mod
B$, there is a monomorphism $B_B\to M^r$ in $\mod B$ for some
positive integer $r$, and hence $\Hom_B(P,M'')\neq 0$ for an
indecomposable direct summand $M''$ of $M$. Since $\mathcal{C}_T$
is a generalized standard component, we infer that then there is
in $\mathcal{C}_T$ a path from $P$ to $M''$. Therefore, any path
in $\mathcal{C}_T$ from $M'$ to $P$ is
a subpath of a path in $\mathcal{C}_T$ from $M'$ to $M''$, and so is sectional, by (2). \\

(4) Let $M''$ be an indecomposable direct summand of $M$ which is
a successor of an indecomposable injective module $I$ in
$\mathcal{C}_T$. Then
every path in $\mathcal{C}_T$ from $I$ to $M''$ is sectional. This follows by arguments dual to those applied in (3). \\

We denote by $\Delta_T$ the section of $\mathcal{C}_T$ given by
the images of a complete set of pairwise nonisomorphic
indecomposable injective $H$-modules via the functor
$\Hom_H(T,-): \mod H \to \mod B$. \\

(5) Let $M_1, M_2, \ldots, M_t$ be a complete set of pairwise
nonisomorphic indecomposable direct summands of $M$. We know that
for a given module $N$ in $\mathcal{C}_T$ there exists a unique
integer $r$ such that $\tau^r_BN\in\Delta_T$. Let $r_1, r_2,
\ldots, r_t$ be the unique integers such that
$\tau^{r_i}_BM_i\in\Delta_T$, for any $i\in\{1,\ldots,t\}$.
Observe that the modules $\tau^{r_1}_BM_1, \tau^{r_2}_BM_2,
\ldots, \tau^{r_t}_BM_t$ are pairwise different because, by (2),
every path in $\mathcal{C}_T$ from $M_i$ to $M_j$, with $i\neq j$
in $\{1,\ldots,t\}$, is sectional. We shall prove our claim by
induction on the number $s(\Delta_T)=\sum_{i=1}^t|r_i|$.

Assume $s(\Delta_T)=0$. Then, for any $i\in\{1,\ldots,t\}$, $M_i\in \Delta_T$ and there is nothing to show.

Assume $s(\Delta_T)\geq 1$. Fix $i\in\{1,\ldots,t\}$ with
$|r_i|\neq 0$. Assume that $r_i>0$, or equivalently,
$M_i\in\mathcal{C}_T\cap\mathcal{X}(T)$. Denote by
$\Sigma^{(i)}_T$ the set of all modules $X$ in $\Delta_T$ such
that there is a path in $\mathcal{C}_T$ of length greater than or
equal to zero from $X$ to $\tau_B^{r_i}M_i$. We note that every
path from a module $X$ in $\Sigma^{(i)}_T$ to $\tau_B^{r_i}M_i$ is
sectional, because $\Delta_T$ is convex in $\mathcal{C}_T$ and
intersects every $\tau_B$-orbit in $\mathcal{C}_T$ exactly once.
Further, by (2) and (4), no module in $\Sigma^{(i)}_T$ is a
successor of a module $M_j$ with $j\in\{1,\ldots,t\}\setminus
\{i\}$ nor an indecomposable injective module, because there is a
nonsectional path in $\mathcal{C}_T$ from $\tau_B^{r_i}M_i$ to
$M_i$. Consider now the full subquiver $\Delta^{(i)}_T$ of
$\mathcal{C}_T$ given by the modules from
$\tau^-_B(\Sigma^{(i)}_T)$ and $\Delta_T\setminus \Sigma^{(i)}_T$.
Then $\Delta^{(i)}_T$ is a section of $\mathcal{C}_T$ and
$s(\Delta^{(i)}_T)\leq s(\Delta_T)-1$.

Assume now $r_i<0$, or equivalently,
$M_i\in\mathcal{C}_T\cap\mathcal{Y}(T)$. Denote by
$\Omega^{(i)}_T$ the set of all modules $Y$ in $\Delta_T$ such
that there is a path in $\mathcal{C}_T$ of length greater than or
equal to zero from $\tau_B^{r_i}M_i$ to $Y$. It follows from (2)
and (3) that no module in $\Omega^{(i)}_T$ is a predecessor of a
module $M_j$ with $j\in\{1,\ldots,t\}\setminus \{i\}$ nor an
indecomposable projective module, because there is a nonsectional
path in $\mathcal{C}_T$ from $M_i$ to $\tau_B^{r_i}M_i$. Consider
now the full subquiver $\Delta^{(i)}_T$ of $\mathcal{C}_T$ given
by the modules from $\tau_B(\Omega^{(i)}_T)$ and
$\Delta_T\setminus \Omega^{(i)}_T$. Then $\Delta^{(i)}_T$ is a
section of $\mathcal{C}_T$ and $s(\Delta^{(i)}_T)\leq
s(\Delta_T)-1$.

Summing up, we obtain that there is a section $\Delta$ in
$\mathcal{C}_T$ containing all modules $M_1, M_2, \ldots, M_t$.
\end{proof}

We complete now the proof of Theorem \ref{thm 1.1}.\\

Let $B$ be an indecomposable tilted algebra and $M$ a sincere
module in $\mod B$ which is not the middle of a short chain in
$\mod B$. Applying Propositions \ref{prop 4.1} and \ref{prop 4.2},
we conclude that there exists a hereditary algebra $\overline{H}$
and a tilting module $\overline{T}$ in $\mod \overline{H}$ such
that $B=\End_{\overline{H}}(\overline{T})$ and $M$ is isomorphic
to a $B$-module $M_1^{n_1} \oplus ... \oplus M_t^{n_t}$ with
$M_1,..., M_t$ indecomposable modules in
$\mathcal{C}_{\overline{T}}$, for some positive integers
$n_1,..,n_t$. Further, it follows from Proposition \ref{prop 4.3}
that there is a section $\Delta$ in $\mathcal{C}_{\overline{T}}$
containing the modules $M_1, ..., M_t$. Denote by
$T^{\ast}_{\Delta}$ the direct sum of all indecomposable
$B$-modules lying on $\Delta$. Then it follows from Theorem
\ref{thm 3.3} that $H_{\Delta}=\End_B(T^{\ast}_{\Delta})$ is an
indecomposable hereditary algebra, $T_{\Delta}=D
(T^{\ast}_{\Delta})$ is a tilting module in $\mod H_{\Delta}$, and
the tilted algebra $B_{\Delta} =\End_{H_{\Delta}}(T_{\Delta})$ is
the basic algebra of $B$. Let $H=H_{\Delta}$. Then there exists a
tilting module $T$ in the additive category $\add(T_{\Delta})$ of
$T_{\Delta}$ in $\mod H =\mod H_{\Delta}$ such that
$B=\End_{H}(T)$, $\mathcal{C}_{\overline{T}}$ is the connecting
component $\mathcal{C}_T$ of $\Gamma_B$ determined by $T$, and
$\Delta$ is the section $\Delta_T$ of $\mathcal{C}_T$ given by the
images of a complete set of pairwise nonisomorphic indecomposable
injective $H$-modules via the functor $\Hom_{H}(T,-): \mod H
\rightarrow \mod B$. Since $M_1,...,M_t$ lie on
$\Delta=\Delta_{T}$, we conclude that there is an injective module
$I$ in $\mod H$ such that the right $B$-modules $M=M_1^{n_1}
\oplus ... \oplus M_t^{n_t}$ and $\Hom_{H}(T,I)$ are isomorphic.
This finishes the proof of Theorem \ref{thm 1.1}.\\

We provide now the proof of Corollary \ref{cor 1.2}.\\

Let $A$ be an algebra and $M$ a module in $\mod A$ which is not
the middle of a short chain. It follows from Theorem \ref{thm 1.1}
that there exists a hereditary algebra $H$ and a tilting module
$T$ in $\mod H$ such that the tilted algebra $B=\End_H(T)$ is a
quotient algebra of $A$ and $M$ is isomorphic to the right
$B$-module $\Hom_B(T,I)$. Further, the functor $\Hom_H(T,-): \mod
H \rightarrow \mod B$ induces an equivalence of the torsion part
$\mathcal{T}(T)$ of $\mod H$ with the torsion-free part
$\mathcal{Y}(T)$ of $\mod B$, and obviously $I$ belongs to
$\mathcal{T}(T)$. Then we obtain isomorphisms of algebras
\[\End_A(M) \cong \End_B(M) \cong \End_B(\Hom_H(T,I)) \cong \End_H(I).\]
Thus Corollary \ref{cor 1.2} follows from the following known
characterization of hereditary algebras (see \cite{KSZ} for more
general results in this direction).
\begin{prop}
Let $\Lambda$ be an algebra. The following conditions are
equivalent:
\begin{itemize}
\item[$(i)$] $\Lambda$ is a hereditary algebra;
\item[$(ii)$] $\End_{\Lambda}(P)$ is a hereditary algebra for any
projective module $P$ in $\mod \Lambda$;
\item[$(iii)$] $\End_{\Lambda}(I)$ is a hereditary algebra for any
injective  module $I$ in $\mod \Lambda$.
\end{itemize}
\end{prop}


\section{\normalsize Examples}

In this section we exhibit examples of modules which are not the middle of short chains, illustrating Theorem \ref{thm 1.1}.\\

\noindent {\sc Example 5.1 \,} Let $K$ be a field, $n$ a positive
integer, $Q$ the quiver
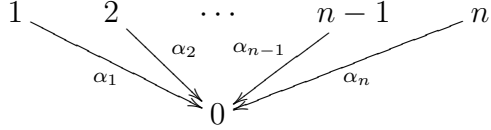
\[
  \xymatrix@C=2pc@R=2pc{
1\ar[rrd]_{\alpha_1} & 2\ar[rd]^{\alpha_2} & \cdots & n-1\ar[ld]_{\alpha_{n-1}} & n\ar[lld]^{\alpha_n} \\
&&0
  }
\]
and $A=KQ$ the path algebra of $Q$ over $K$. Then the Auslander-Reiten quiver $\Gamma_A$ admits a unique preinjective component $\mathcal{Q}(A)$ whose
right part is of the form
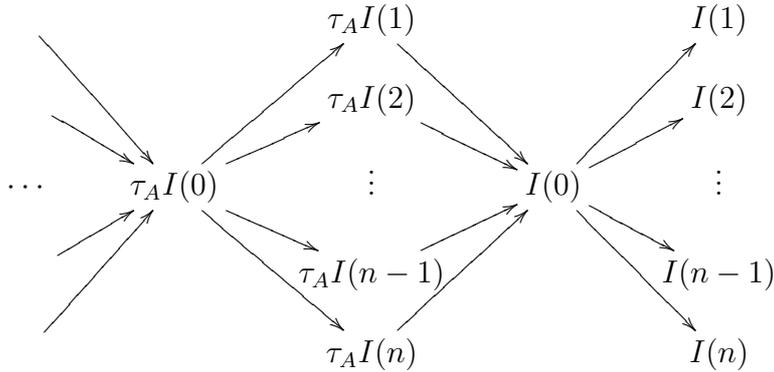
\[
  \xymatrix@C=2pc@R=1pc{
{\phantom{III}}\ar[rdd]&&\tau_AI(1)\ar[rdd]&&I(1) \\
{\phantom{III}}\ar[rd]&&\tau_AI(2)\ar[rd]&&I(2) \\
\cdots&\tau_AI(0)\ar[ruu]\ar[ru]\ar[rdd]\ar[rd]&\vdots&I(0)\ar[ruu]\ar[ru]\ar[rdd]\ar[rd]&\vdots \\
{\phantom{III}}\ar[ru]&&\tau_AI(n-1)\ar[ru]&&I(n-1) \\
{\phantom{III}}\ar[ruu]&&\tau_AI(n)\ar[ruu]&&I(n)
  }
\]
where $I(0), I(1), I(2) \ldots, I(n-1), I(n)$ are the
indecomposable injective right $A$-modules at the vertices $0, 1,
2, \ldots, n-1, n$, respectively. Consider the semisimple module
$M=I(1)\oplus I(2)\oplus\ldots\oplus I(n-1)\oplus I(n)$ in $\mod
A$. Then $M$ is not the middle of a short chain in $\mod A$ and
$B=A/\ann_A(M)$ is the path algebra $K\Delta$ of the subquiver
$\Delta$ of $Q$ given by the vertices $1, 2, \ldots, n-1, n$,
which is isomorphic to the product of $n$ copies of $K$.
Observe also that the injective modules $I(0), I(1),..., I(n)$ form a section of $\mathcal{Q}(A)$.\\

\noindent {\sc Example 5.2 \,} Let $K$ be a field and $n$ be a
positive integer. For each $i\in\{1,\ldots,n\}$, choose a basic
indecomposable finite-dimensional hereditary $K$-algebra $H_i$, a
multiplicity-free tilting module $T_i$ in $\mod H_i$, and consider
the associated tilted algebra $B_i=\End_{H_i}(T_i)$, the
connecting component $\mathcal{C}_{T_i}$ of $\Gamma_{B_i}$
determined by $T_i$, and the module
$M_{T_i}=\Hom_{H_i}(T_i,D(H_i))$ whose indecomposable direct
summands form the canonical section $\Delta_{T_i}$ of
$\mathcal{C}_{T_i}$. It follows from general theory (\cite{K1},
\cite{St}) that the Auslander-Reiten quiver $\Gamma_{B_i}$
contains at least one preinjective component. Therefore, we may
choose, for any $i\in\{1,\ldots,n\}$, a simple injective right
$B_i$-module $S_i$ lying in a preinjective component
$\mathcal{Q}_i$ of $\Gamma_{B_i}$. Let $B=B_1\times\ldots\times
B_n$ and $S=S_1\oplus\ldots\oplus S_n$. Then $S$ is a
finite-dimensional $K$-$B$-bimodule, and we may consider the
one-point extension algebra
\[ A=\left[\begin{matrix}K&S\\ 0 & B\end{matrix}\right] = \left\{\left[\begin{matrix}\lambda &s\\0&b
\end{matrix}\right]\mid \,\,\lambda\in K,\,\, s\in S,\,\,b\in B\right\}. \]
Since $S$ is a semisimple injective module in $\mod B$, it follows
from general theory (see \cite[(2.5)]{Ri1} or \cite[(XV.1)]{SS2})
that, for any indecomposable module $X$ in $\mod A$ which is not
in $\mod B$, its radical $\rad X$ coincides with the largest right
$B$-submodule of $X$ and belongs to the additive category
$\add(S)$ of $S$. In particular, for any indecomposable module $Z$
in $\mod B$, the almost split sequence in $\mod B$ with the right
term $Z$ is an almost split sequence in $\mod A$. Therefore, the
Auslander-Reiten quiver $\Gamma_A$ of $A$ is obtained from the
disjoint union of the Auslander-Reiten quiver $\Gamma_{B_1}, ...,
\Gamma_{B_n}$ by glueing of the preinjective components
$\mathcal{Q}_1, ..., \mathcal{Q}_n$ into a one component by the
new indecomposable projective  $A$-module $P$ with $\rad P=S$ (and
possibly adding new components). This implies that the right
$B$-module
\[ M=M_{T_1}\oplus\ldots\oplus M_{T_n} \]
is not the middle of a short chain in $\mod A$. Moreover, since $M_{T_i}$ is a faithful right $B_i$-module for any $i\in\{1,\ldots,n\}$, we
conclude that $B=A/\ann_A(M)$.\\

\noindent {\sc Example 5.3 \,} Let $K$ be a field, $H$ be a basic
indecomposable finite-dimensional hereditary $K$-algebra, $T$ a
multiplicity-free  tilting module in $\mod H$, and $B=\End_H(T)$
the associated tilted algebra. For a positive integer $r\geq 2$,
consider the $r$-fold trivial extension algebra
\[
  T(B)^{(r)} =
  \left\{\begin{array}{c}
   \left[\begin{array}{cccccc}
    b_1 & 0 & 0 & & & \\
    f_2 & b_2 & 0 & & 0 & \\
    0 & f_3 & b_3 & & & \\
     &  & \ddots & \ddots & & \\
     & 0 &  & f_{r-1} & b_{r-1} & 0 \\
     &  &  & 0 & f_1 & b_1 \\
   \end{array}\right] \\
   b_1,\dots,b_{r-1} \in B, \ f_1,\dots,f_{r-1} \in D(B) \\
  \end{array}\right\}
\]
of $B$. Then $T(B)^{(r)}$ is a basic indecomposable
finite-dimensional selfinjective $K$- algebra which is isomorphic
to the orbit algebra $\widehat{B} / (\nu_{\widehat{B}}^r)$ of the
repetitive algebra $\widehat{B}$ of $B$ with respect to the
infinite cyclic group $(\nu_{\widehat{B}}^r)$ of automorphisms of
$\widehat{B}$ generated by the $r$-th power of the Nakayama
automorphism $\nu_{\widehat{B}}$ of $\widehat{B}$. Moreover, we
have the canonical Galois covering $F^{(r)}: \widehat{B}\to
\widehat{B} / (\nu_{\widehat{B}}^r) = T(B)^{(r)}$ and the
associated push-down functor $F_{\lambda}^{(r)}: \mod\widehat{B}
\to \mod T(B)^{(r)}$ is dense (see \cite[Sections 6 and 7]{SY} for
more details). We also note that $T(B)^{(r)}$ admits a quotient
algebra $B_1\times B_2\times\ldots\times B_{r-1}$ with $B_i=B$ for
any $i\in\{1,2,\ldots,r-1\}$.

Fix a positive integer $m$ and consider the selfinjective algebra $A_m=T(B)^{(4(m+1))}$. For each $i\in\{1,2,\ldots,m\}$, consider the quotient algebra
$B_{4i}=B$ of $A_m$ and the right $B_{4i}$-module $M_{4i}=\Hom_{H}(T,D(B))$, being the direct sum of all indecomposable modules lying on the canonical
section $\Delta_{4i}=\Delta_T$ of the connecting component $\mathcal{C}_{4i}=\mathcal{C}_{T}$ of $\Gamma_{B_{4i}}$ determined by $T$. Then, applying
arguments as in \cite[Section 2]{RSS1}, we conclude that
\[ M=\bigoplus_{i=1}^mM_{4i} \]
is a module in $\mod A_m$ which is not the middle of a short chain and $A_m/\ann_{A_m}(M)$ is isomorphic to the product
\[ \prod_{i=1}^mB_{4i} \]
of $m$ copies of the tilted algebra $B$.

\bigskip

\begin{center}\footnotesize
Faculty of Mathematics and Computer Science, Nicolaus Copernicus University, \\ Chopina 12/18, 87-100 Toru\'n, Poland \\
e-mail: jaworska@mat.uni.torun.pl \\
e-mail: pmalicki@mat.uni.torun.pl \\
e-mail: skowron@mat.uni.torun.pl \\
\end{center}

\end{document}